\documentclass[12pt]{amsart}
\usepackage{amssymb,amsmath,amsthm,amsfonts,color,longtable}
\usepackage{multirow}
\usepackage{longtable}
\usepackage{hyperref}

\newcommand{\bea}{\begin{eqnarray}}
\newcommand{\eea}{\end{eqnarray}}
\newcommand{\beaa}{\begin{eqnarray*}}
\newcommand{\eeaa}{\end{eqnarray*}}
\newcommand{\g}{\mathfrak g}
\newcommand{\h}{\mathfrak h}
\newcommand{\n}{\mathfrak n}
\newcommand{\Z}{\mathbb Z}

\newcommand{\C}{\mathbb C}

\newcommand{\sreg}{\mathrm{sreg}}
\usepackage[margin=2.5cm]{geometry} 
\usepackage{comment}
\newtheorem{theorem}{Theorem}[section]
\newtheorem{lemma}[theorem]{Lemma}
\newtheorem{conj}[theorem]{Conjecture}
\newtheorem{proposition}[theorem]{Proposition}

\newtheorem{corollary}[theorem]{Corollary}
\newtheorem{remark}[theorem]{Remark}

\numberwithin{theorem}{section}
\numberwithin{equation}{section}
\newcommand{\DS}{\mathrm{DS}}

\newcommand{\W}{\mathcal W}

\newcommand{\Sf}{\mathcal S_f}

\newcommand{\ea}[1]{e_{#1}}
\newcommand{\fa}[1]{f_{#1}}
\newcommand{\hh}[1]{h_{#1}}

\allowdisplaybreaks

\begin{document}

\title[]{  A new quasi-lisse affine vertex algebra of type $D_4$ }

\author[]{Dra\v zen  Adamovi\' c}
\address{Department of Mathematics, Faculty of Science \\
	University of Zagreb \\
	Bijeni\v cka 30 }
\email{adamovic@math.hr}

\author[]{Ivana Vukorepa}
\address{Faculty of Science \\
	University of Split \\
	Ru\dj era Bo\v{s}kovi\'ca 33}
\email{ivukorepa@pmfst.hr}

 \begin{abstract}
	We consider a family of potential quasi-lisse affine vertex algebras $L_{k_m}(D_4)$ at levels  $k_m =-6 + \frac{4}{2m+1}$.   In the case $m=0$,    the irreducible   $L_{k_0}(D_4)$--modules  were classified in  \cite{P-2013},  and it was proved in \cite{ArK} that $L_{k_0}(D_4)$ is a quasi-lisse vertex algebra. We conjecture  that $L_{k_m}(D_4)$ is quasi-lisse for every $m \in {\Z}_{>0}$, and that it contains a unique irreducible ordinary module. 
	
	In this article we prove this conjecture for $m=1$, by using mostly computational methods. We show that  the maximal ideal in the universal affine vertex algebra $V^{k_1}(D_4)$ is generated by three singular vectors of conformal weight six. The explicit formulas were obtained using software. Then we apply Zhu's theory and classify all irreducible $L_{k_1}(D_4)$--modules. It turns out that 
	$L_{k_1}(D_4)$ has $405$ irreducible modules in the category $\mathcal O$, but a unique irreducible ordinary module. Finally, we prove that $L_{k_1}(D_4)$ is quasi-lisse	by showing that  its  associated variety   is contained in the nilpotent cone of $D_4$. {We also prove that the associated variety \(X_{L_{k_1}(D_4)}\) is
\(\overline{\mathbb O}_{\sreg}\), the Zariski closure of the subregular
nilpotent orbit in \(D_4\).}

	 %
\end{abstract}
\maketitle

\section{Introduction}
 
 Let $L_k(\g)$ be an affine vertex algebra of level $k$ associated to the Lie algebra $\g$. 
 Recall that $L_k(\g)$  belongs to the Cvitanovi\' c-Deligne exceptional series of vertex algebras  if $\g$ is of type
\[
A_1 \subset A_2 \subset G_2 \subset D_4  {
	 \subset F_4} \subset E_6 \subset E_7 \subset E_8,
\]
and  the level is given by the formula   $k = -\frac{h^\vee}{6} - 1$.
 This  series of vertex algebras has attracted a lot of interest in physics since  they have origin in 4d  SCFT. In mathematics, these vertex algebras can be characterized by the property that their minimal quantum Hamiltonian reduction  is trivial, i.e, $k$ is collapsing level  and $W_k(\g, f_{min}) = {\C} {\bf 1}$ (cf. \cite{AKMPP18, ArM}). Moreover, these vertex  algebras are quasi-lisse (cf. \cite{ArK}) and their associated variety is  $\overline{{\Bbb O}}_{min}$ (cf. \cite{ArM}).
 
 It is a natural task to find new examples of vertex algebras with similar properties.  In paper  \cite{Ar-M-Dai}  the authors {study in detail} the affine vertex algebra $L_{-2}(G_2)$, prove that it is quasi-lisse,  and   identify the associated variety  $X_{L_{-2} (G_2)}$ with $\overline{{\Bbb O}}_{sreg}$ (=the Zariski closure of the subregular  nilpotent orbit in $G_2$).
 The paper  \cite{Ar-M-Dai} shows how difficult it  is to find new examples of quasi-lisse affine vertex algebras. Singular vector  which was found  in  \cite{Ar-M-Dai}  contains $385$ terms in the PBW basis, and 
 {working with such
a singular vector is much more complicated}
  than in the case of the exceptional series. However,   we will show that one can find new examples of quasi-lisse VOAs by using expressions for singular vectors in the PBW bases.

 In this article we study vertex algebras $L_{k_m}(D_4)$ at levels  $k_m =-6 + \frac{4}{2m+1}$. In the case $m=0$, the vertex algebra  $L_{k_0}(D_4)= L_{-2}(D_4)$ was first studied by  O. Perše in   \cite{P-2013}. He proved that
 \begin{itemize}
\item The maximal ideal in the universal affine vertex algebra   $V^{-2} (D_4)$ is generated by three singular  vectors of conformal weight $2$.
\item The simple VOA $L_{-2}(D_4)$ {has a unique}, up to isomorphism, irreducible ordinary module.
\item $L_{-2} (D_4)$ has $5$ irreducible modules in the category $\mathcal O$. 
 \end{itemize}
 
 
 {In the present paper, we study the vertex algebras} $L_{k_m}(D_4)$.  We believe that the following results hold:
 
 \begin{conj} \label{slutnja-1}
\item[(1)]  $L_{k_m}(D_4)$ is a quasi-lisse vertex algebra for each $m \ge 1$.
\item[(2)] $L_{k_m}(D_4)$ {has a unique irreducible ordinary} $L_{k_m}(D_4)$--module and $KL_{k_m}(D_4)$ {is a semisimple} braided tensor category.
 \end{conj}
 
 Our main result says:
 \begin{theorem} \label{main}
 The conjecture \ref{slutnja-1} holds for $m=1$.
 \end{theorem}

 We   also note that alternatively one can show that the maximal ideal in $V^{k_m}(D_4)$ is generated by three singular vectors of  conformal weight  $2 (2m+1)$  following the methods  developed by C. Jiang and  J. Song  \cite{JS}. Their methods use the structural results from {Fiebig's} paper \cite{F}  which in our cases says that the character formula of $L_{k_m}(D_4)$ depends only on the numerator, which is in all cases equal to $4$.  Since the maximal ideal of $V^{k_0}(D_4)$ is generated by three singular  vectors (cf. \cite{P-2013}) we get that the same holds for {the entire series}. {Analogous methods} can be also used for the determination of maximal ideal for certain non-admissible affine  vertex algebras of types $E$  coming from Cvitanovi\'c-Deligne series.

{We should point out that  we do not use these methods in the case $k=1$,}
since we explicitly show that the quotient $\widetilde L_{k_1}(D_4)$ of $V^{k_1}(D_4)$  by ideal $J$  generated by three singular vectors has only one irreducible module in $KL_{k_1}$ (cf. Corollary \ref{KL_k}) which directly implies that the maximal ideal of $V^{k_1}(D_4)$  (cf. Theorem \ref{quasi-lisse} ) is generated by these singular vectors.
 
 The basic computational tasks of our methods are based on the following:
 \begin{itemize}
 \item Computation of {an} explicit formula for {a} singular vector in $V^{k_1}(D_4)$. The singular vector is obtained by very delicate computer calculations. {The}  explicit formula for 
 {the} singular vector in the PBW basis contains $6753$ monomials, and it is presented in the Mathematica file on more than $200$ pages. By applying a certain automorphism of vertex algebra $V^{-14/3}(D_4)$ (cf. Remark \ref{auto}) we obtain two more singular vectors in $V^{-14/3}(D_4)$. 
 \item Determination of the projection of one singular vector in Zhu's algebra and application of automorphism from Remark \ref{auto}, which results with nine polynomial equations from which we calculate all irreducible highest weight modules in the category $\mathcal O$. As a consequence, we classify modules in the Kazhdan-Lusztig category. 
 \item Calculation of projection of one singular vector in Zhu's  $C_2$--algebra and application of automorphism from Remark \ref{auto}, which also  produces nine (different) polynomial equations implying that the associated variety $X_{L_{-\frac{14}{3}}(D_4)}$ does not contain any non-zero semi-simple element. As a consequence we conclude that  $X_{L_{-\frac{14}{3}}(D_4)}$ is  contained in the nilpotent cone of $D_4$. This proves that $L_{-\frac{14}{3}}(D_4)$ is quasi-lisse.

 \end{itemize}
 
 The vertex operator  algebra $L_{k_1}(D_4)$ appeared in  \cite[Table 21]{LXY},   indicating  that    it is an example of  a  vertex operator algebra  coming from $4d$ $N=2$ SCFT.
The following conjecture predicts a more precise result on the associated variety.  Let $H^0_{DS, f}(-)$  denote the Drinfeld–Sokolov reduction with respect to some  nilpotent element $f$, and $W_k(\g,f)$ be the simple quotient of the universal affine vertex algebra $W^k(\g, f) = H^0_{DS, f}(V^k(\g))$ (cf. \cite{KW04}).
 Recall also  that the subregular nilpotent element $f_{sreg}$  of $D_4$ corresponds to the partition $[5,3]$. {We prove the following result.\footnote{This result was conjectured in the first arxiv version of the paper.}}

 \begin{theorem}  \label{associated-variety}We have: 
 \begin{itemize}
  \item[(1)]  $W_{k_1} (D_4, f_{sreg}) = H^0_{DS, f_{sreg}} (L_{k_1} (D_4)) = {\C} {\bf 1}$.
 \item[(2)]  $X_{L_{k_1}(D_4)} =  \overline{{\Bbb O}}_{sreg}$.
 
 \end{itemize}
 
   \end{theorem}
 The proof of Theorem~\ref{associated-variety} is presented in Section~\ref{proof-thm1.3}. We first show that the conformal-weight two subspace of the universal affine $W$--algebra $W^{k_1}(D_4,f_{sreg})$ is three-dimensional. Then, by analyzing the Drinfeld--Sokolov reduction functor $H^0_{DS,f}$, we show that the three singular vectors of conformal weight $6$ in $V^{k_1}(D_4)$ are mapped to three linearly independent conformal-weight two generators of $W^{k_1}(D_4,f_{sreg})$. As a consequence, the ideal $I=H^0_{DS,f}(J)$ of $W^{k_1}(D_4,f_{sreg})$ contains the Virasoro vector $L$. This then implies that all strong generators of $W^{k_1}(D_4,f_{sreg})$ belong to $I$, and therefore vanish in the simple quotient. This gives assertion (1) of the theorem. Assertion (2) is a consequence of (1) and standard arguments about nilpotent orbits and associated varieties. 
 
 This result also shows that $k_1$ is a collapsing level for $f_{sreg}$ which  extends some recent results on collapsing levels from \cite{AA+}.

\subsection*{Acknowledgements}
We would like to thank   Antun Milas,  Ozren Per\v se, Paolo Papi  and  Wenbin Yan for useful discussion.  {We especially thank Tomoyuki Arakawa and the referee for discussions concerning the proof of Theorem \ref{associated-variety}.}
The authors are partially supported by the Croatian Science Foundation under the project IP-2022-10-9006 
 and by the project “Implementation of cutting-edge research and its application as part of the Scientific Center of Excellence QuantiXLie“, PK.1.1.02, European Union, European Regional Development Fund.

\section{Preliminaries}\label{sec-prelim}

\subsection{Affine vertex algebras}
Let $\mathfrak{g}$ be a simple Lie algebra with a triangular decomposition 
$\mathfrak{g} = \mathfrak{n}_{-} \oplus \mathfrak{h} \oplus \mathfrak{n}_{+}$.
Denote by $\left( \cdot \, | \, \cdot \right)$ the invariant bilinear form on $\mathfrak{g}$, normalized by the condition $\left( \theta \, | \, \theta \right)=2$, where $\theta$ is the highest root of $\mathfrak{g}$. We will often identify $\mathfrak{h}^*$ with $\mathfrak{h}$ via $\left( \cdot \, | \, \cdot \right)$. 
For $\mu \in \mathfrak{h}^*$,  let $V(\mu )$  denote the irreducible highest weight $\mathfrak{g}$--module with highest weight $\mu$. Let   $\alpha_1, \ldots, \alpha_{\ell}$ denote  simple roots, and let  $h_1, \ldots, h_{\ell}$  be simple co-roots ($h_i= \alpha_{i}^{\vee}$, for $i=1, \ldots , \ell$). We shall denote by  $\omega _1 , \ldots, \omega_{\ell}$ {fundamental} weights for $\mathfrak{g}$ where $\ell = \dim \mathfrak{h}$. Let $\hat{\mathfrak{g}}$ be the affine Kac-Moody Lie algebra associated to $\mathfrak{g}$. Let  $\alpha_0, \alpha_1, \ldots, \alpha_{\ell}$ be simple roots, $\alpha_{0}^{\vee}, \alpha_{1}^{\vee} \ldots, \alpha_{\ell}^{\vee}$ simple co-roots, and $\Lambda _0 , \Lambda _1, \ldots, \Lambda_{\ell}$ {fundamental} weights for $\hat{\mathfrak{g}}$. For $\mu \in \mathfrak{h}^*$ and $k \in \mathbb{C}$, denote by $L_{k}\left( \mu \right)$ the irreducible highest weight $\hat{\mathfrak{g}}$--module with highest weight $\widehat{\mu} : = k \Lambda _0 + \mu  \in \hat{\mathfrak{h}} ^*$. 

Denote by $V^{k}(\mathfrak{g})$ the universal affine vertex algebra associated to simple Lie algebra $\mathfrak{g}$ and level $k \in \mathbb{C}$, $k \neq -h^{\vee}$, and 
by $L_{k}(\mathfrak{g})$ the unique simple quotient of $V^{k}(\mathfrak{g})$.

\subsection{Zhu's theory}\label{subsec-Zhu} For a vertex    algebra $V$, denote by $A(V)$ Zhu’s algebra associated to $V$ (cf. \cite{Zhu96}). For $m \in \Z_{>0}$, let $J = \langle v_1, \ldots , v_m \rangle$ be the ideal in $V^k(\g)$ generated by singular vectors $v_1, \ldots , v_m \in V^k(\g)$ and let $\widetilde{L}_k(\g) = V^k(\g) / J$ be the associated quotient vertex algebra.
%
{There exists a method for classifying irreducible}
 $\widetilde{L}_k(\g)$--modules in the category $\mathcal{O}$ developed in \cite{A-94,AM-95}, very important for us as it allows us to obtain classification once we know singular vectors. We recall it just briefly here since this method has been {discussed} in previous papers \cite{AP-08, APV-21} for $m=1$ and in the paper \cite{P-2013} for the case $\g = D_l$. 

From \cite{ FZ-92, Zhu96} we have that $A(V^k(\g)) \cong \mathcal{U}(\g)$ and $A(\widetilde{L}_k(\g)) \cong \mathcal{U}(\g) / I$, where $I$ is the two-sided ideal in $\mathcal{U}(\g)$ generated by vectors $v_1', \ldots , v_m'$, where $v_i' \in \mathcal{U}(\g)$ is the image of $v_i \in V^k(\g)$ in $\mathcal{U}(\g)$. For $i \in \{ 1, \ldots , m\}$, let $R^{(i)}$ be a $\mathcal{U}(\g)$--submodule of $\mathcal{U}(\g)$ generated by $v_i'$ under adjoint action. We denote by $R^{(i)}_0$ its zero-weight space. Let $V(\mu)$ be an irreducible highest weight $\mathcal{U}(\g)$–module with the highest weight vector $v_\mu$, for $\mu \in \h^*$. Clearly, for each $r \in R^{(i)}_0$ there exists the unique polynomial $p_r \in \mathcal{S}(\h)$ such that $r v_\mu = p_r(\mu) v_\mu$. Set \begin{equation}\label{polinomi}
\mathcal{P}^{(i)}_0 = \{ p_r \in \mathcal{S}(\h) \ | \ r\in R^{(i)}_0 \}.
\end{equation} We have:

\begin{proposition}\label{prop-Zhu} There is a one-to-one correspondence between
	\begin{itemize}
		\item[(1)] irreducible $\widetilde{L}_k(\g)$–modules in the category $\mathcal{O}$ (for $\hat{\g}$),
		\item[(2)] irreducible $A(\widetilde{L}_k(\g))$–modules in the category $\mathcal{O}$ (for ${\g}$),
		\item[(3)] weights $\mu \in \h^*$ such that $p(\mu) = 0$ for all $p \in \mathcal{P}^{(i)}_0$ and all $i \in \{ 1, \ldots , m\}$.
	\end{itemize}
\end{proposition}

\subsection{Zhu's $C_2$--algebra} \label{subsect-C2-zhu} 

For a vertex operator  algebra $V$, denote by $R_V = V/C_2(V)$ Zhu's $C_2$--algebra of $V$ (cf. \cite{Zhu96}). Then 
$$ R_{V^{k}(\mathfrak{g})} \cong \mathcal{S}(\mathfrak{g}) $$
under the algebra isomorphism uniquely determined by 
\begin{equation} \label{C2-map}
	\overline{x(-1)\mathbf{1}} \mapsto x, \quad \mbox{for} \ x \in \mathfrak{g}, 
\end{equation}
where $\overline{w}$ denotes the image of $w$ in $R_{V^{k}(\mathfrak{g})}$, for any $w \in V^{k}(\mathfrak{g})$.  Let $v$ be a $\hat{\mathfrak{g}}$--singular vector in $V^{k}(\mathfrak{g})$. Denote by
$\langle v \rangle$ the ideal in $V^{k}(\mathfrak{g})$ generated by $v$ and by
\begin{equation} \label{L-tilde}
	\widetilde{L}_{k}\left( \mathfrak{g}\right)=V^{k}(\mathfrak{g})/ \langle v \rangle
\end{equation}

For the quotient vertex operator algebra $\widetilde{L}_{k}( \mathfrak{g})$,
 denote by $v''$ the image of vector $\overline{v}$ under the map (\ref{C2-map}). Then 
$$ R_{\widetilde{L}_{k}( \mathfrak{g})} \cong \mathcal{S}(\mathfrak{g}) / I_{W},  $$
where $W$ is a $\mathfrak{g}$--module generated by $v''$ under the adjoint action, and $I_W$ is the ideal of $\mathcal{S}(\mathfrak{g})$ generated by $W$ (cf. \cite{ArM}, \cite{ArMo-sheets}).

We consider    $\mathcal{S}(\mathfrak{g})$ as $\mathfrak h$--module.  Then  we have the following decomposition
\begin{equation} \label{dec-11}  \mathcal{S}(\mathfrak{g})  =   \mathcal{S}(\mathfrak h) \oplus {\mathfrak n}_- \mathcal{S}(\mathfrak{g})  \oplus \mathcal{S}(\mathfrak{g})  \mathfrak n_+.  \end{equation} 
Let  $\mathcal{S}(\mathfrak{g}) ^{\mathfrak h} = \{ v \in  \mathcal{S}(\mathfrak{g}) \ \vert \    \ [h, v] =0, \ \forall h \in \mathfrak h\}$.  
Following \cite{ArM},  we consider  the  Chevalley projection map $\Psi :   \mathcal{S}(\mathfrak{g}) ^{\mathfrak h} \rightarrow \mathcal{S}(\mathfrak h)$, 
 with respect to the decomposition (\ref{dec-11}).

\section{Affine vertex algebra associated to $D_4$ at level $-14/3$}
In this section let $\g$ denotes the Lie algebra of type $D_4$.  We fix the root vectors
for $\g$ as in \cite{Bou}, \cite{FF}  and  \cite{P-2013}. 
 More precisely, let $A = A_1 \oplus A_2 $ where $A_1$ is a vector space with basis $\{ a_i \ | \ i=1,\ldots,4\}$ and $A_2$ a vector space with dual basis $\{ a_i^* \ | \ i=1,\ldots,4 \}$, so we have $$ [a_i, a_j]_+ = [a_i^* , a_j^*]_+ = 0, \ \ [a_i, a_j^*]_+ = \delta_{ij}, \ \ i,j = 1, \ldots ,4. $$
Defining normal ordering on $A$ by $: xy: = \frac{1}{2}(xy-yx)$, for $x,y \in A$, we have that all normally ordered quadratic elements $:xy:$ span a Lie algebra of type $D_4$ with Cartan subalgebra $\h$ spanned by $H_i=:a_i a_i^* :$, for $i =1, \ldots ,4$ \cite{Bou, FF}. Let $\{ \varepsilon_i \ | \ i=1, \ldots ,4 \}$ be the basis of $\h^*$ dual to $\{H_i \ |\ i=1, \ldots ,4\}$, so we have $\varepsilon_{i}(H_j) = \delta_{ij}$. The root system of $\g$ is given by $$ \Delta = \{  \pm (\varepsilon_{i} \pm \varepsilon_{j})  \ | \  1 \leq i< j \leq 4 \},$$ with a set of simple roots given by $\alpha_i = \varepsilon_{i} - \varepsilon_{i+1}$, for $i=1,2,3$ and $\alpha_4 = \varepsilon_{3} + \varepsilon_{4} $. The set of positive roots is  $\{ \varepsilon_{i} \pm \varepsilon_{j}  \ | \  1 \leq i< j \leq 4 \}$. Now we fix root vectors $$ e_{\varepsilon_{i} - \varepsilon_{j} } = :a_i a_j^*: , \ \ e_{\varepsilon_{i} + \varepsilon_{j} } = :a_i a_j: ,\ \ f_{\varepsilon_{i} - \varepsilon_{j} } = :a_j a_i^*: , \ \ f_{\varepsilon_{i} + \varepsilon_{j} } = :a_j^* a_i^*: ,$$
for $1 \leq i < j \leq 4$.

\begin{theorem}\label{sing-v}
There is a singular vector $v$ in $V^{-14/3}(\g)$ of weight $-\frac{14}{3} \Lambda_0 - 6 \delta + 2 \omega_1$. Its
explicit formula is given in Mathematica file \textit{\href{https://www.dropbox.com/scl/fi/d9o8d0urjdzaske0l9yp1/D4-sing-v.pdf?rlkey=cdn83hob1qem6u83qypozs9xb&st=a7gqo1z9&dl=0}{D4-sing-v.nb}}.
\end{theorem}

\begin{proof}
	Direct verification of relations $e_{\varepsilon_i - \varepsilon_{i+1}} (0) v = 0 $ for $i = 1,2,3$, $e_{\varepsilon_3 + \varepsilon_{4}} (0) v = 0$ and $f_{\varepsilon_1 + \varepsilon_{2}} (1) v = 0$ using Mathematica.
\end{proof}

\begin{remark}\label{auto}
	Vertex algebra $V^{-14/3}(\g)$ has an order three automorphism $\sigma$ lifted from the automorphism of the Dynkin diagram of $ D_4$ such that \begin{equation}
		\sigma(\varepsilon_1 - \varepsilon_2)= \varepsilon_3 - \varepsilon_4, \ 	\sigma(\varepsilon_2 - \varepsilon_3)= \varepsilon_2 - \varepsilon_3, \ 	\sigma(\varepsilon_3 - \varepsilon_4)= \varepsilon_3 + \varepsilon_4, \ 	\sigma(\varepsilon_3 + \varepsilon_4)= \varepsilon_1 - \varepsilon_2. 
	\end{equation}
\end{remark}

Using automorphism $\sigma$ from Remark \ref{auto} we get two more singular vectors $\sigma(v)$ and $\sigma^2 (v)$ in $V^{-14/3}(\g)$. Let us denote $$ \widetilde{L}_{-14/3} (\g) = V^{-14/3} (\g) / J,$$
where $J$ is the ideal in $V^{-14/3} (D_4)$ generated by singular vectors $v$, $\sigma(v)$ and $\sigma^2(v)$.

From Zhu's theory recalled in Subsection (\ref{subsec-Zhu}) we have that Zhu's algebra of $ \widetilde{L}_{-14/3} (D_4)$ is isomorphic to $\mathcal{U}(D_4) / \langle v', \sigma(v'), \sigma^2(v') \rangle$, where $v' \in \mathcal{U}(D_4)$ is the image of singular vector $v$ from Theorem \ref{sing-v}. The explicit formula of $v'$ is given in Mathematica file \textit{\href{https://www.dropbox.com/scl/fi/ccf29o6ahp6joerbiflc6/D4-category-O.pdf?rlkey=pbrwdck9lwyjsegx963qg5gys&st=k4h2zlie&dl=0}{D4-category-O.nb}}. For vector $v'$, we denote by $\mathcal{P}_0^{(1)}$ the vector space of polynomials defined by (\ref{polinomi}).

\begin{lemma}The basis of the vector space $\mathcal{P}_0^{(1)}$ is given by the set $\{ p_1, \ p_2, \ p_3 \}$ where
		{\footnotesize
		\begin{eqnarray*}
			p_1 && \hspace*{-6mm}= h_1 (14 + 3 h_1 + 6 h_2 + 3 h_3 + 3 h_4) (3200 + 6160 h_1 + 4260 h_1^2 + 1260 h_1^3 + 135 h_1^4 + 7568 h_2 + 9192 h_1 h_2 +\\&& 3924 h_1^2 h_2 + 540 h_1^3 h_2 + 6420 h_2^2 + 4320 h_1 h_2^2 + 864 h_1^2 h_2^2 + 2376 h_2^3 + 648 h_1 h_2^3 + 324 h_2^4 + 2560 h_3 + 3840 h_1 h_3 +\\&& 1800 h_1^2 h_3 + 270 h_1^3 h_3 + 5232 h_2 h_3 + 3996 h_1 h_2 h_3 + 864 h_1^2 h_2 h_3 + 3240 h_2^2 h_3 + 972 h_1 h_2^2 h_3 + 648 h_2^3 h_3 +\\&& 480 h_3^2 + 540 h_1 h_3^2 + 135 h_1^2 h_3^2 + 864 h_2 h_3^2 + 324 h_1 h_2 h_3^2 + 324 h_2^2 h_3^2 + 2560 h_4 + 3840 h_1 h_4 + 1800 h_1^2 h_4 +\\&& 270 h_1^3 h_4 + 5232 h_2 h_4 + 3996 h_1 h_2 h_4 + 864 h_1^2 h_2 h_4 + 3240 h_2^2 h_4 + 972 h_1 h_2^2 h_4 + 648 h_2^3 h_4 + 1536 h_3 h_4 +\\&& 1836 h_1 h_3 h_4 + 432 h_1^2 h_3 h_4 + 2916 h_2 h_3 h_4 + 972 h_1 h_2 h_3 h_4 + 972 h_2^2 h_3 h_4 + 108 h_3^2 h_4 + 162 h_1 h_3^2 h_4 +\\&& 324 h_2 h_3^2 h_4 + 480 h_4^2 + 540 h_1 h_4^2 + 135 h_1^2 h_4^2 + 864 h_2 h_4^2 + 324 h_1 h_2 h_4^2 + 324 h_2^2 h_4^2 + 108 h_3 h_4^2 + 162 h_1 h_3 h_4^2 +\\&& 324 h_2 h_3 h_4^2 - 81 h_3^2 h_4^2),\\
			p_2 && \hspace*{-6mm}= h_3 h_4 (5120 - 1008 h_1 - 5508 h_1^2 - 2268 h_1^3 - 243 h_1^4 + 17424 h_2 + 7128 h_1 h_2 - 2916 h_1^2 h_2 - 972 h_1^3 h_2 + 17820 h_2^2 +\\&& 7776 h_1 h_2^2 + 7128 h_2^3 + 1944 h_1 h_2^3 + 972 h_2^4 + 5760 h_3 + 1296 h_1 h_3 - 1944 h_1^2 h_3 - 486 h_1^3 h_3 + 14256 h_2 h_3 +\\&& 6804 h_1 h_2 h_3 + 9720 h_2^2 h_3 + 2916 h_1 h_2^2 h_3 + 1944 h_2^3 h_3 + 1440 h_3^2 + 324 h_1 h_3^2 - 243 h_1^2 h_3^2 + 2592 h_2 h_3^2 +\\&& 972 h_1 h_2 h_3^2 + 972 h_2^2 h_3^2 + 5760 h_4 + 1296 h_1 h_4 - 1944 h_1^2 h_4 - 486 h_1^3 h_4 + 14256 h_2 h_4 + 6804 h_1 h_2 h_4 + \\&&
			9720 h_2^2 h_4 + 2916 h_1 h_2^2 h_4 + 1944 h_2^3 h_4 + 6480 h_3 h_4 + 2916 h_1 h_3 h_4 + 8748 h_2 h_3 h_4 + 2916 h_1 h_2 h_3 h_4 + \\&&2916 h_2^2 h_3 h_4 + 1620 h_3^2 h_4 + 486 h_1 h_3^2 h_4 + 972 h_2 h_3^2 h_4 + 1440 h_4^2 + 324 h_1 h_4^2 - 243 h_1^2 h_4^2 + 2592 h_2 h_4^2 + \\&&972 h_1 h_2 h_4^2 + 972 h_2^2 h_4^2 + 1620 h_3 h_4^2 + 486 h_1 h_3 h_4^2 + 972 h_2 h_3 h_4^2 + 405 h_3^2 h_4^2), \\
			p_3 && \hspace*{-6mm}= h_2 (9856 + 19152 h_1 + 9396 h_1^2 + 2268 h_1^3 + 243 h_1^4 + 27856 h_2 + 47016 h_1 h_2 + 18144 h_1^2 h_2 + 3240 h_1^3 h_2 + 243 h_1^4 h_2 +\\&& 29700 h_2^2 + 40824 h_1 h_2^2 + 10692 h_1^2 h_2^2 + 972 h_1^3 h_2^2 + 14940 h_2^3 + 14904 h_1 h_2^3 + 1944 h_1^2 h_2^3 + 3564 h_2^4 + 1944 h_1 h_2^4 +\\&& 324 h_2^5 + 16768 h_3 + 30816 h_1 h_3 + 13284 h_1^2 h_3 + 2754 h_1^3 h_3 + 243 h_1^4 h_3 + 34992 h_2 h_3 + 51516 h_1 h_2 h_3 +\\&& 14580 h_1^2 h_2 h_3 + 1458 h_1^3 h_2 h_3 + 26316 h_2^2 h_3 + 
			27864 h_1 h_2^2 h_3 + 3888 h_1^2 h_2^2 h_3 + 8424 h_2^3 h_3 + 4860 h_1 h_2^3 h_3 +\\&& 972 h_2^4 h_3 + 8064 h_3^2 + 13284 h_1 h_3^2 + 4131 h_1^2 h_3^2 + 486 h_1^3 h_3^2 + 12528 h_2 h_3^2 + 14580 h_1 h_2 h_3^2 + 2187 h_1^2 h_2 h_3^2 +\\&& 6156 h_2^2 h_3^2 + 3888 h_1 h_2^2 h_3^2 + 972 h_2^3 h_3^2 + 1152 h_3^3 + 1620 h_1 h_3^3 + 243 h_1^2 h_3^3 + 1296 h_2 h_3^3 + 972 h_1 h_2 h_3^3 + 324 h_2^2 h_3^3 +\\&& 16768 h_4 + 30816 h_1 h_4 + 13284 h_1^2 h_4 + 2754 h_1^3 h_4 + 243 h_1^4 h_4 + 34992 h_2 h_4 + 51516 h_1 h_2 h_4 + 14580 h_1^2 h_2 h_4 +\\&& 1458 h_1^3 h_2 h_4 + 26316 h_2^2 h_4 + 27864 h_1 h_2^2 h_4 + 	3888 h_1^2 h_2^2 h_4 + 8424 h_2^3 h_4 + 4860 h_1 h_2^3 h_4 + 972 h_2^4 h_4 + 
			28800 h_3 h_4 +\\&& 42444 h_1 h_3 h_4 + 11664 h_1^2 h_3 h_4 + 972 h_1^3 h_3 h_4 + 41580 h_2 h_3 h_4 + 42768 h_1 h_2 h_3 h_4 + 5832 h_1^2 h_2 h_3 h_4 +\\&& 19440 h_2^2 h_3 h_4 + 10692 h_1 h_2^2 h_3 h_4 + 2916 h_2^3 h_3 h_4 + 
			12852 h_3^2 h_4 + 15066 h_1 h_3^2 h_4 + 2187 h_1^2 h_3^2 h_4 + 12636 h_2 h_3^2 h_4 +\\&& 7290 h_1 h_2 h_3^2 h_4 + 2916 h_2^2 h_3^2 h_4 + 1620 h_3^3 h_4 + 1458 h_1 h_3^3 h_4 + 972 h_2 h_3^3 h_4 + 8064 h_4^2 + 13284 h_1 h_4^2 + 4131 h_1^2 h_4^2 +\\&& 486 h_1^3 h_4^2 + 12528 h_2 h_4^2 + 14580 h_1 h_2 h_4^2 + 2187 h_1^2 h_2 h_4^2 + 6156 h_2^2 h_4^2 + 
			3888 h_1 h_2^2 h_4^2 + 972 h_2^3 h_4^2 + 12852 h_3 h_4^2 +\\&& 
			15066 h_1 h_3 h_4^2 + 2187 h_1^2 h_3 h_4^2 + 12636 h_2 h_3 h_4^2 + 7290 h_1 h_2 h_3 h_4^2 + 2916 h_2^2 h_3 h_4^2 + 4131 h_3^2 h_4^2 + 2916 h_1 h_3^2 h_4^2 +\\&& 2187 h_2 h_3^2 h_4^2 + 243 h_3^3 h_4^2 + 1152 h_4^3 + 1620 h_1 h_4^3 + 243 h_1^2 h_4^3 + 1296 h_2 h_4^3 + 972 h_1 h_2 h_4^3 + 324 h_2^2 h_4^3 + 1620 h_3 h_4^3 +\\&& 1458 h_1 h_3 h_4^3 + 972 h_2 h_3 h_4^3 + 243 h_3^2 h_4^3).
\end{eqnarray*}	}
\end{lemma}

\begin{proof}$\mathcal{U}(\g)$--submodule $R^{(1)}$ of $\mathcal{U}(\g)$ generated by vector $v'$ under the adjoint action is clearly isomorphic to $V(2 \omega_1)$. From \cite{P-2013} we have that its zero-weight subspace $R_0^{(1)}$ has dimension $3$.  By direct calculation one obtains:\begin{eqnarray*}
		(f_{\varepsilon_1-\varepsilon_{2}}  f_{\varepsilon_{1}+\varepsilon_{2}})_L v' \in p_1 (h) + \mathcal{U}(D_4) \n_+,\\
		(f_{\varepsilon_1-\varepsilon_{4}}  f_{\varepsilon_{1}+\varepsilon_{4}} - f_{\varepsilon_1-\varepsilon_{3}}  f_{\varepsilon_{1}+\varepsilon_{3}})_L v' \in p_2 (h) + \mathcal{U}(D_4) \n_+,\\
		(f_{\varepsilon_1-\varepsilon_{2}}  f_{\varepsilon_{1}+\varepsilon_{2}} - f_{\varepsilon_1-\varepsilon_{3}}  f_{\varepsilon_{1}+\varepsilon_{3}})_L v' \in p_3 (h) + \mathcal{U}(D_4) \n_+.
	\end{eqnarray*}
It is easy to see that polynomials $p_1, \ p_2$ and $p_3$ are linearly independent. The proof follows. 
\end{proof}

\begin{proposition}\label{cat-O} The complete list of irreducible $\widetilde{L}_{-14/3}(\g)$--modules in the category $\mathcal{O}$ consists of 405 weights given in the following table: 
	
	\begin{center}\label{table-catO}
			\begin{longtable}{ |p{9cm}|p{7cm}| } 
		\hline
		$0, -\frac{14}{3} \omega_i, \ -\frac{10}{3} \omega_i, \ -\frac{8}{3} \omega_i, \ -2 \omega_i,\ -\frac{4}{3} \omega_i, \ \ i=1,3,4$ & $-\frac{11}{3} \omega_2, \ -\frac{8}{3} \omega_2, \ -\frac{7}{3} \omega_2, \  -\frac{4}{3} \omega_2, \ - \omega_2$  \\ 
		\hline
		\multicolumn{2}{|p{16cm}|}{$\frac{8}{3} \omega_i - \frac{11}{3} \omega_2, \ 
			\pm \frac{2}{3} \omega_i - \frac{8}{3} \omega_2, \ 
			\frac{4}{3} \omega_i - \frac{8}{3} \omega_2, \ 
			\pm \frac{4}{3} \omega_i - \frac{7}{3} \omega_2,\  
			-2 \omega_i - \frac{5}{3} \omega_2,  \
			\pm \frac{2}{3} \omega_i - \frac{5}{3} \omega_2, \ 
			-\frac{4}{3} \omega_i - \frac{4}{3} \omega_2, \newline 
			-\frac{2}{3} \omega_i -\frac{4}{3} \omega_2, \  
			-2 \omega_i - \frac{4}{3} \omega_2, \ 
			-\frac{8}{3} \omega_i -  \omega_2, \ 
			-\frac{4}{3} \omega_i -  \omega_2, \ 
			-\frac{10}{3} \omega_i - \frac{1}{3} \omega_2, \ 
			-\frac{2}{3} \omega_i - \frac{1}{3} \omega_2, \ \ i=1,3,4$} \\ 
		\hline
		\multicolumn{2}{|p{16.2cm}|}{$
			\pm \frac{2}{3} \omega_{i_1} - \frac{8}{3} \omega_{i_2}, \ 
			- \frac{2}{3} \omega_{i_1} - \frac{4}{3} \omega_{i_2}, \
			- \frac{4}{3} \omega_{i_1} - \frac{4}{3} \omega_{i_2}, \
			- \frac{4}{3} \omega_{i_1} - 2\omega_j, \
			- \frac{4}{3} \omega_{i_1} - \frac{8}{3} \omega_{i_2}, \newline
			- \frac{10}{3} \omega_{i_1} - \frac{4}{3} \omega_{i_2}, \
			- 2 \omega_{i_1} - \frac{8}{3} \omega_{i_2}, \
			- \frac{8}{3} \omega_{i_1} - \frac{8}{3} \omega_{i_2}, \ 
			(i_1,i_2) \in \{  (1,3),(1,4),(3,1),(3,4),(4,1),(4,3) \}$} \\ 
		\hline
		\multicolumn{2}{|p{16cm}|}{$\frac{2}{3} \omega_{i_1} - \frac{10}{3} \omega_{2} + \frac{2}{3} \omega_{i_2}, \ 
			-\frac{2}{3} \omega_{i_1} - 2 \omega_{2} - \frac{2}{3} \omega_{i_2}, \ 
			-\frac{2}{3} \omega_{i_1} - \frac{5}{3} \omega_{2} - \frac{2}{3} \omega_{i_2}, \ 
			\frac{2}{3} \omega_{i_1} - \frac{5}{3} \omega_{2} + \frac{2}{3} \omega_{i_2}, \newline 
			-\frac{4}{3} \omega_{i_1} - \frac{4}{3} \omega_{2} - \frac{4}{3} \omega_{i_2}, \ 
			-\frac{4}{3} \omega_{i_1} -  \omega_{2} - \frac{4}{3} \omega_{i_2}, \ 
			-\frac{2}{3} \omega_{i_1} - \frac{2}{3} \omega_{2} - \frac{2}{3} \omega_{i_2}, \ 
			-\frac{2}{3} \omega_{i_1} - \frac{1}{3} \omega_{2} - \frac{2}{3} \omega_{i_2}, \newline 
			-\frac{4}{3} \omega_{i_1} + \frac{1}{3} \omega_{2} - \frac{4}{3} \omega_{i_2}, \ 
			-\frac{8}{3} \omega_{i_1} + \frac{5}{3} \omega_{2} - \frac{8}{3} \omega_{i_2}, \ 
			(i_1,i_2) \in \{  (1,3),(1,4),(3,4) \} $}
		\\
		\hline
		\multicolumn{2}{|p{16cm}|}{$
			-\frac{2}{3} \omega_{i_1} - \frac{8}{3} \omega_{2} + \frac{4}{3} \omega_{i_2}, \ 
			-\frac{4}{3} \omega_{i_1} - \frac{7}{3} \omega_{2} + \frac{4}{3} \omega_{i_2}, \ 
			-\frac{2}{3} \omega_{i_1} - 2 \omega_{2} +\frac{2}{3} \omega_{i_2}, \newline
			-\frac{8}{3} \omega_{i_1} - \frac{5}{3} \omega_{2} + \frac{2}{3} \omega_{i_2}, \  
			-\frac{4}{3} \omega_{i_1} - \frac{5}{3} \omega_{2} \pm \frac{2}{3} \omega_{i_2}, \ 
			-\frac{2}{3} \omega_{i_1} - \frac{5}{3} \omega_{2} + \frac{2}{3} \omega_{i_2}, \ 
			-2 \omega_{i_1} - \frac{5}{3} \omega_{2} + \frac{2}{3} \omega_{i_2}, \newline 
			-\frac{4}{3} \omega_{i_1} - \frac{4}{3} \omega_{2} \pm \frac{2}{3} \omega_{i_2}, \ 
			-2 \omega_{i_1} - \frac{2}{3} \omega_{2} - \frac{2}{3} \omega_{i_2}, \ 
			-\frac{8}{3} \omega_{i_1} - \frac{1}{3} \omega_{2} - \frac{2}{3} \omega_{i_2}, \ 
			-2 \omega_{i_1} - \frac{1}{3} \omega_{2} - \frac{4}{3} \omega_{i_2}, \newline 
			-2 \omega_{i_1} - \frac{1}{3} \omega_{2} - \frac{2}{3} \omega_{i_2}, \ 
			-\frac{4}{3} \omega_{i_1} - \frac{1}{3} \omega_{2} - \frac{2}{3} \omega_{i_2}, \ 
			-\frac{10}{3} \omega_{i_1} - \frac{1}{3} \omega_{2} - \frac{2}{3} \omega_{i_2}, \ 
			-\frac{8}{3} \omega_{i_1} + \frac{1}{3} \omega_{2} - \frac{4}{3} \omega_{i_2}, \newline 
			(i_1,i_2) \in \{  (1,3),(1,4),(3,1),(3,4),(4,1),(4,3) \}$} \\ 
		\hline
		\multicolumn{2}{|p{16cm}|}{$
			- \frac{4}{3} \omega_1 - \frac{4}{3} \omega_3 - \frac{4}{3} \omega_4, \ 
			- \frac{5}{3} \omega_1 - \frac{5}{3} \omega_3 - \frac{5}{3} \omega_4, \ 
			- \frac{8}{3} \omega_{i_1} - \frac{8}{3} \omega_{i_2} + \frac{2}{3} \omega_{i_3}, \ 
			- \frac{7}{3} \omega_{i_1} - \frac{1}{3} \omega_{i_2} - \frac{1}{3} \omega_{i_3}, \newline 
			- 2 \omega_{i_1} - \frac{4}{3} \omega_{i_2} - \frac{4}{3} \omega_{i_3}, \ 
			- \frac{5}{3} \omega_{i_1} - \frac{5}{3} \omega_{i_2} -  \omega_{i_3}, \ 
			- \frac{5}{3} \omega_{i_1} - \frac{5}{3} \omega_{i_2} \pm \frac{1}{3} \omega_{i_3}, \ 
			- \frac{5}{3} \omega_{i_1} - \frac{1}{3} \omega_{i_2} - \frac{1}{3} \omega_{i_3}, \newline 
			- \frac{4}{3} \omega_{i_1} - \frac{4}{3} \omega_{i_2} - \frac{2}{3} \omega_{i_3}, \ 
			- \frac{8}{3} \omega_{i_1} - \frac{4}{3} \omega_{i_2} - \frac{2}{3} \omega_{i_3}, \ 
			- \frac{8}{3} \omega_{i_1} - \frac{2}{3} \omega_{i_2} - \frac{4}{3} \omega_{i_3}, 
			- \frac{7}{3} \omega_{i_1} - \frac{5}{3} \omega_{i_2} - \frac{1}{3} \omega_{i_3}, \newline  
			- \frac{7}{3} \omega_{i_1} - \frac{1}{3} \omega_{i_2} - \frac{5}{3} \omega_{i_3}, \ 
			- \frac{5}{3} \omega_{i_1} -  \omega_{i_2} - \frac{1}{3} \omega_{i_3}, \ 
			- \frac{5}{3} \omega_{i_1} - \frac{1}{3} \omega_{i_2} - \omega_{i_3}, \newline
			(i_1,i_2,i_3) \in \{  (1,3,4),(3,4,1),(4,1,3) \}$} \\ 
		\hline
		\multicolumn{2}{|p{16cm}|}{$
			\frac{2}{3} \omega_1 - \frac{10}{3} \omega_2 + \frac{2}{3} \omega_3 + \frac{2}{3} \omega_4, \
			-\frac{1}{3} \omega_1 - 2 \omega_2 - \frac{1}{3} \omega_3 - \frac{1}{3} \omega_4, \
			-\frac{1}{3} \omega_1 - \frac{4}{3} \omega_2 - \frac{1}{3} \omega_3 - \frac{1}{3} \omega_4, \newline
			-\frac{2}{3} \omega_1 - \frac{2}{3} \omega_2 - \frac{2}{3} \omega_3 - \frac{2}{3} \omega_4, \
			-\frac{2}{3} \omega_1 -  \omega_2 - \frac{2}{3} \omega_3 - \frac{2}{3} \omega_4, \
			-\frac{1}{3} \omega_1 - \frac{2}{3} \omega_2 - \frac{1}{3} \omega_3 - \frac{1}{3} \omega_4, \newline
			-\frac{5}{3} \omega_1 + \frac{2}{3} \omega_2 - \frac{5}{3} \omega_3 - \frac{5}{3} \omega_4, \
			-\frac{5}{3} \omega_1 + \frac{4}{3} \omega_2 - \frac{5}{3} \omega_3 - \frac{5}{3} \omega_4, \
			-\frac{4}{3} \omega_1 + \frac{1}{3} \omega_2 - \frac{4}{3} \omega_3 - \frac{4}{3} \omega_4, \newline
			-\frac{8}{3} \omega_1 + \frac{5}{3} \omega_2 - \frac{8}{3} \omega_3 - \frac{8}{3} \omega_4
			$} \\
		\hline
		\multicolumn{2}{|p{16cm}|}{$
			\frac{2}{3} \omega_{i_1} - 2 \omega_{2} - \frac{2}{3} \omega_{i_2} - \frac{2}{3} \omega_{i_3}, \
			\frac{1}{3} \omega_{i_1} - 2 \omega_{2} - \frac{1}{3} \omega_{i_2} - \frac{1}{3} \omega_{i_3}, \
			-\frac{8}{3} \omega_{i_1} - \frac{5}{3} \omega_{2} + \frac{2}{3} \omega_{i_2} + \frac{2}{3} \omega_{i_3}, \newline
			-\frac{5}{3} \omega_{i_1} - \frac{4}{3} \omega_{2} - \frac{1}{3} \omega_{i_2} - \frac{1}{3} \omega_{i_3}, \
			\frac{2}{3} \omega_{i_1} - \frac{4}{3} \omega_{2} - \frac{4}{3} \omega_{i_2} - \frac{4}{3} \omega_{i_3}, \
			- \omega_{i_1} - \frac{4}{3} \omega_{2} - \frac{1}{3} \omega_{i_2} - \frac{1}{3} \omega_{i_3}, \newline
			\frac{1}{3} \omega_{i_1} - \frac{4}{3} \omega_{2} - \frac{1}{3} \omega_{i_2} - \frac{1}{3} \omega_{i_3}, \
			-\frac{7}{3} \omega_{i_1} - \frac{2}{3} \omega_{2} - \frac{1}{3} \omega_{i_2} - \frac{1}{3} \omega_{i_3}, \
			-2 \omega_{i_1} - \frac{2}{3} \omega_{2} - \frac{2}{3} \omega_{i_2} - \frac{2}{3} \omega_{i_3}, \newline
			-\frac{5}{3} \omega_{i_1} - \frac{2}{3} \omega_{2} - \frac{1}{3} \omega_{i_2} - \frac{1}{3} \omega_{i_3}, \
			- \omega_{i_1} - \frac{2}{3} \omega_{2} - \frac{1}{3} \omega_{i_2} - \frac{1}{3} \omega_{i_3}, \
			-\frac{4}{3} \omega_{i_1} - \frac{2}{3} \omega_{2} - \frac{2}{3} \omega_{i_2} - \frac{2}{3} \omega_{i_3}, \newline
			-\frac{8}{3} \omega_{i_1} - \frac{1}{3} \omega_{2} - \frac{2}{3} \omega_{i_2} - \frac{2}{3} \omega_{i_3}, \
			-\frac{2}{3} \omega_{i_1} - \frac{1}{3} \omega_{2} - \frac{4}{3} \omega_{i_2} - \frac{4}{3} \omega_{i_3}, \
			-\frac{4}{3} \omega_{i_1} - \frac{1}{3} \omega_{2} - \frac{2}{3} \omega_{i_2} - \frac{2}{3} \omega_{i_3}, \newline
			-\frac{8}{3} \omega_{i_1} + \frac{1}{3} \omega_{2} - \frac{4}{3} \omega_{i_2} - \frac{4}{3} \omega_{i_3}, \
			-\frac{7}{3} \omega_{i_1} + \frac{2}{3} \omega_{2} - \frac{5}{3} \omega_{i_2} - \frac{5}{3} \omega_{i_3}, \
			- \omega_{i_1} + \frac{2}{3} \omega_{2} - \frac{5}{3} \omega_{i_2} - \frac{5}{3} \omega_{i_3}, \newline
			-\frac{1}{3} \omega_{i_1} + \frac{2}{3} \omega_{2} - \frac{5}{3} \omega_{i_2} - \frac{5}{3} \omega_{i_3}, \
			-\frac{7}{3} \omega_{i_1} + \frac{4}{3} \omega_{2} - \frac{5}{3} \omega_{i_2} - \frac{5}{3} \omega_{i_3}, \newline (i_1,i_2,i_3) \in \{  (1,3,4), (3,4,1), (4,1,3)  \}
			$} \\
		\hline
		\multicolumn{2}{|p{16cm}|}{$
			-\frac{4}{3} \omega_{i_1} - \frac{5}{3} \omega_{2} + \frac{2}{3} \omega_{i_2} - \frac{2}{3} \omega_{i_3}, \
			-\frac{5}{3} \omega_{i_1} - \frac{4}{3} \omega_{2} + \frac{1}{3} \omega_{i_2} - \frac{1}{3} \omega_{i_3}, \
			-\frac{5}{3} \omega_{i_1} - \frac{2}{3} \omega_{2} - \frac{1}{3} \omega_{i_2} - \frac{5}{3} \omega_{i_3}, \newline
			-\frac{5}{3} \omega_{i_1} - \frac{2}{3} \omega_{2} - \frac{1}{3} \omega_{i_2} -  \omega_{i_3}, \
			-\frac{5}{3} \omega_{i_1} - \frac{2}{3} \omega_{2} + \frac{1}{3} \omega_{i_2} - \frac{1}{3} \omega_{i_3}, \
			-2 \omega_{i_1} - \frac{1}{3} \omega_{2} - \frac{2}{3} \omega_{i_2} - \frac{4}{3} \omega_{i_3}, \newline
			-\frac{7}{3} \omega_{i_1} + \frac{2}{3} \omega_{2} - \frac{1}{3} \omega_{i_2} - \frac{5}{3} \omega_{i_3}, \ \ (i_1,i_2,i_3) \in \{  (1,3,4), (1,4,3), (3,1,4), (3,4,1), (4,1,3), (4,3,1)  \}
			$} \\
		\hline
	\end{longtable}
	\end{center}
\end{proposition}
\begin{proof}
For the singular vectors $\sigma(v)$ and $\sigma^2 (v)$, by $R^{(2)}$ and $R^{(3)}$ we denote $\mathcal{U}(\g)$--submodules of $\mathcal{U}(\g)$ generated by their images in Zhu's algebra, respectively. It is easy to see that $R^{(2)} \cong V(2 \omega_3)$, $R^{(3)} \cong V(2 \omega_4)$, and ${\rm{dim} }R_0^{(2)} = {\rm{dim} }R_0^{(3)} = 3$. The basis of vector space $\mathcal{P}_0^{(2)}$ (resp. $\mathcal{P}_0^{(3)}$ ) is given by the set $\{ p_4, \ p_5, \ p_6 \}$  (resp. $\{ p_7, \ p_8, \ p_9 \}$  ), where
polynomials $p_4, \ p_5, \ p_6$ are obtained by applying the automorphism $\sigma$ on polynomials $p_1, \ p_2, \ p_3$ respectively. Similarly, polynomials $p_7, \ p_8, \ p_9$ are obtained by applying the automorphism $\sigma^2$ on polynomials $p_1, \ p_2$ and  $p_3$. Their explicit formulas can be found in Mathematica file \textit{\href{https://www.dropbox.com/scl/fi/ccf29o6ahp6joerbiflc6/D4-category-O.pdf?rlkey=pbrwdck9lwyjsegx963qg5gys&st=k4h2zlie&dl=0}{D4-category-O.nb}}.\\
Now, from Proposition \ref{prop-Zhu}, we have that highest weights $\mu \in \h^*$ of irreducible $\widetilde{L}_{-14/3}(\g)$--modules in the category $\mathcal{O}$ are exactly solutions of polynomial equations $p_1 (\mu(h)) = p_2 (\mu(h)) = \cdots = p_9 (\mu(h)) = 0$. Using Mathematica, one checks that these are exactly the weights given in the Table \ref{table-catO}.
\end{proof}

The only dominant integral weight for $\g$ in the Table \ref{table-catO} is $\mu=0$. Thus, we have:

\begin{corollary}\label{KL_k} $L_{-14/3}(\g)$ is the unique irreducible module for $\widetilde{L}_{-14/3}(\g)$ in $KL_{-14/3}$.
\end{corollary}

\begin{theorem} \label{quasi-lisse} We have:
	\begin{itemize}
		\item[(i)] $J$ is the maximal ideal in $V^{-14/3}(\g)$, i.e. $L_{-14/3}(\g) = \widetilde{L}_{-14/3}(\g)$.
		\item[(ii)] The simple affine vertex algebra $L_{-14/3}(\g)$ is quasi-lisse.
	\end{itemize}
\end{theorem}

\begin{proof} If $\widetilde{L}_{-14/3}(\g)$ is not simple, the Corollary \ref{KL_k} implies that any non-trivial singular vector in $\widetilde{L}_{-14/3}(\g)$ must be proportional to $\bf{1}$. This proves (i). 
	
	Next we prove that the associated variety of $L_{-14/3}(\g)$ is contained in the nilpotent cone of $\g$. 
	We get $ R_{L_{-14/3}}(\g) \cong  S(\g)  / I_W  $, where $I_W$ is    the ideal defined  in Subsection  \ref{subsect-C2-zhu}  from three projections $v''$, $\sigma(v'')$, $\sigma^2(v'')$ of singular vectors $v$, $\sigma(v)$ and $\sigma^2(v)$, respectively.  Explicit formula for vector $v''$ is given in Mathematica file \textit{\href{https://www.dropbox.com/scl/fi/o39z1jd2aesnxino0k7gr/D4-v.pdf?rlkey=nklzr66t13s6gz57n6get8c0x&st=4jyxd6yd&dl=0}{D4-v''.nb}}. We set
	$$  I_W ^{\mathfrak  h}  = \Psi ( I_W \cap   S(\g)^{\mathfrak h}), $$
	where $\Psi$ is the  Chevalley projection map.

	Now we shall construct nine linearly independent polynomials of $\mathfrak h$ which belong to $I_W ^{\mathfrak  h}$. We first have  $q_1, q_2, q_3 \in  I_W ^{\mathfrak  h}$:
		\begin{eqnarray*}
			&&[f_{\varepsilon_1 - \varepsilon_2},[f_{\varepsilon_1 + \varepsilon_2},v'']] \in q_1 + \mathcal{S}(\n_+), \\&&
			[f_{\varepsilon_1 - \varepsilon_3},[f_{\varepsilon_1 + \varepsilon_3},v'']] \in q_2 + \mathcal{S}(\n_+), \\&& 
			[f_{\varepsilon_1 - \varepsilon_4},[f_{\varepsilon_1 + \varepsilon_4},v'']] \in q_3 + \mathcal{S}(\n_+),
		\end{eqnarray*}
		
		where{\footnotesize
			\begin{eqnarray*}
				q_1(h_1, h_2, h_3, h_4)  && \hspace*{-6mm}= 81 h_1 (h_1 + 2 h_2 + h_3 + h_4) (5 h_1^4 + 20 h_1^3 h_2 + 32 h_1^2 h_2^2 + 24 h_1 h_2^3 + 12 h_2^4 + 10 h_1^3 h_3 + 32 h_1^2 h_2 h_3 \\ && + 36 h_1 h_2^2 h_3 + 24 h_2^3 h_3 + 5 h_1^2 h_3^2 + 12 h_1 h_2 h_3^2 + 12 h_2^2 h_3^2 + 10 h_1^3 h_4 + 32 h_1^2 h_2 h_4 + 36 h_1 h_2^2 h_4 + 24 h_2^3 h_4 \\&&+ 16 h_1^2 h_3 h_4 + 36 h_1 h_2 h_3 h_4 + 36 h_2^2 h_3 h_4 + 6 h_1 h_3^2 h_4 + 12 h_2 h_3^2 h_4 + 5 h_1^2 h_4^2 + 12 h_1 h_2 h_4^2 + 12 h_2^2 h_4^2 \\&& + 6 h_1 h_3 h_4^2 + 12 h_2 h_3 h_4^2 - 3 h_3^2 h_4^2),\\
				q_2(h_1, h_2, h_3, h_4) && \hspace*{-6mm}=  81 (h_1 + h_2) (h_1 + h_2 + h_3 + h_4) (5 h_1^4 + 20 h_1^3 h_2 + 24 h_1^2 h_2^2 + 8 h_1 h_2^3 - 4 h_2^4 + 10 h_1^3 h_3 \\&&+ 24 h_1^2 h_2 h_3 + 12 h_1 h_2^2 h_3 - 8 h_2^3 h_3 + 5 h_1^2 h_3^2 + 4 h_1 h_2 h_3^2 - 4 h_2^2 h_3^2 + 10 h_1^3 h_4 + 24 h_1^2 h_2 h_4 + 12 h_1 h_2^2 h_4 \\&& - 8 h_2^3 h_4 + 16 h_1^2 h_3 h_4 + 20 h_1 h_2 h_3 h_4 - 20 h_2^2 h_3 h_4 + 
				6 h_1 h_3^2 h_4 - 12 h_2 h_3^2 h_4 + 5 h_1^2 h_4^2 + 4 h_1 h_2 h_4^2 \\&&- 4 h_2^2 h_4^2 + 6 h_1 h_3 h_4^2 - 12 h_2 h_3 h_4^2 - 3 h_3^2 h_4^2),\\
				q_3(h_1, h_2, h_3, h_4) && \hspace*{-6mm}= 81 (h_1 + h_2 + h_3) (h_1 + h_2 + h_4) (5 h_1^4 + 20 h_1^3 h_2 + 24 h_1^2 h_2^2 + 8 h_1 h_2^3 - 4 h_2^4 + 10 h_1^3 h_3 \\&&+ 24 h_1^2 h_2 h_3 + 12 h_1 h_2^2 h_3 - 8 h_2^3 h_3 + 5 h_1^2 h_3^2 + 4 h_1 h_2 h_3^2 - 4 h_2^2 h_3^2 + 10 h_1^3 h_4 + 24 h_1^2 h_2 h_4 + 12 h_1 h_2^2 h_4 \\&&- 8 h_2^3 h_4 + 8 h_1^2 h_3 h_4 + 4 h_1 h_2 h_3 h_4 - 4 h_2^2 h_3 h_4 - 2 h_1 h_3^2 h_4 + 4 h_2 h_3^2 h_4 + 5 h_1^2 h_4^2 + 4 h_1 h_2 h_4^2 - 
				4 h_2^2 h_4^2\\&& - 2 h_1 h_3 h_4^2 + 4 h_2 h_3 h_4^2 + 5 h_3^2 h_4^2).
		\end{eqnarray*}}
		
		Using automorphisms $\sigma$ and $\sigma^2$ from Remark \ref{auto} we obtain the following polynomials $q_4, \ldots , q_9$ in  $I_W ^{\mathfrak  h}$:
		{\footnotesize
			\begin{eqnarray*}
				q_4 (h_1, h_2, h_3, h_4)  && \hspace*{-6mm}= 81 h_3 (h_1 + 2 h_2 + h_3 + h_4) (12 h_1^2 h_2^2 + 24 h_1 h_2^3 + 12 h_2^4 + 12 h_1^2 h_2 h_3 + 36 h_1 h_2^2 h_3 + 24 h_2^3 h_3 \\ && + 5 h_1^2 h_3^2 + 32 h_1 h_2 h_3^2 + 32 h_2^2 h_3^2 + 10 h_1 h_3^3 + 20 h_2 h_3^3 + 5 h_3^4 + 12 h_1^2 h_2 h_4 + 36 h_1 h_2^2 h_4 + 24 h_2^3 h_4 \\ && + 6 h_1^2 h_3 h_4 + 36 h_1 h_2 h_3 h_4 + 36 h_2^2 h_3 h_4 + 16 h_1 h_3^2 h_4 + 32 h_2 h_3^2 h_4 + 10 h_3^3 h_4 - 3 h_1^2 h_4^2 + 12 h_1 h_2 h_4^2 \\ &&+ 12 h_2^2 h_4^2 + 6 h_1 h_3 h_4^2 + 12 h_2 h_3 h_4^2 + 5 h_3^2 h_4^2) , \\
				q_5  (h_1, h_2, h_3, h_4) && \hspace*{-6mm}= -81 (h_2 + h_3) (h_1 + h_2 + h_3 + h_4) (4 h_1^2 h_2^2 + 8 h_1 h_2^3 + 4 h_2^4 - 	4 h_1^2 h_2 h_3 - 12 h_1 h_2^2 h_3 - 8 h_2^3 h_3 \\ && - 5 h_1^2 h_3^2 - 24 h_1 h_2 h_3^2 - 24 h_2^2 h_3^2 - 10 h_1 h_3^3 - 20 h_2 h_3^3 - 5 h_3^4 + 12 h_1^2 h_2 h_4 + 20 h_1 h_2^2 h_4 + 8 h_2^3 h_4 \\ && - 6 h_1^2 h_3 h_4 - 20 h_1 h_2 h_3 h_4 - 12 h_2^2 h_3 h_4 - 16 h_1 h_3^2 h_4 - 24 h_2 h_3^2 h_4 - 10 h_3^3 h_4 + 3 h_1^2 h_4^2 + 12 h_1 h_2 h_4^2 \\ && + 4 h_2^2 h_4^2 - 6 h_1 h_3 h_4^2 - 4 h_2 h_3 h_4^2 - 5 h_3^2 h_4^2), \\
				q_6 (h_1, h_2, h_3, h_4)  && \hspace*{-6mm}= -81 (h_1 + h_2 + h_3) (h_2 + h_3 + h_4) (4 h_1^2 h_2^2 + 8 h_1 h_2^3 + 4 h_2^4 - 4 h_1^2 h_2 h_3 - 12 h_1 h_2^2 h_3 - 8 h_2^3 h_3 \\ && - 5 h_1^2 h_3^2 - 24 h_1 h_2 h_3^2 - 24 h_2^2 h_3^2 - 10 h_1 h_3^3 - 20 h_2 h_3^3 - 5 h_3^4 - 4 h_1^2 h_2 h_4 + 4 h_1 h_2^2 h_4 + 8 h_2^3 h_4 \\ && + 2 h_1^2 h_3 h_4 - 4 h_1 h_2 h_3 h_4 - 12 h_2^2 h_3 h_4 - 8 h_1 h_3^2 h_4 - 24 h_2 h_3^2 h_4 - 10 h_3^3 h_4 - 5 h_1^2 h_4^2 - 4 h_1 h_2 h_4^2 \\ && + 4 h_2^2 h_4^2 + 2 h_1 h_3 h_4^2 - 4 h_2 h_3 h_4^2 - 5 h_3^2 h_4^2), \\
				q_7 (h_1, h_2, h_3, h_4)  && \hspace*{-6mm}= 81 h_4 (h_1 + 2 h_2 + h_3 + h_4) (12 h_1^2 h_2^2 + 24 h_1 h_2^3 + 12 h_2^4 + 	12 h_1^2 h_2 h_3 + 36 h_1 h_2^2 h_3 + 24 h_2^3 h_3 \\ && - 3 h_1^2 h_3^2 + 12 h_1 h_2 h_3^2 + 12 h_2^2 h_3^2 + 12 h_1^2 h_2 h_4 + 36 h_1 h_2^2 h_4 + 24 h_2^3 h_4 + 6 h_1^2 h_3 h_4 + 36 h_1 h_2 h_3 h_4 \\&& + 36 h_2^2 h_3 h_4  + 6 h_1 h_3^2 h_4 + 12 h_2 h_3^2 h_4 + 5 h_1^2 h_4^2 + 32 h_1 h_2 h_4^2 + 32 h_2^2 h_4^2 + 16 h_1 h_3 h_4^2 + 32 h_2 h_3 h_4^2 \\&& + 5 h_3^2 h_4^2 + 10 h_1 h_4^3 + 20 h_2 h_4^3 + 10 h_3 h_4^3 + 5 h_4^4) , \\
				q_8  (h_1, h_2, h_3, h_4) && \hspace*{-6mm}= -81 (h_2 + h_4) (h_1 + h_2 + h_3 + h_4) (4 h_1^2 h_2^2 + 8 h_1 h_2^3 + 4 h_2^4 + 12 h_1^2 h_2 h_3 + 20 h_1 h_2^2 h_3 + 8 h_2^3 h_3 \\&& 
				+ 3 h_1^2 h_3^2 + 12 h_1 h_2 h_3^2 + 4 h_2^2 h_3^2 - 4 h_1^2 h_2 h_4 - 12 h_1 h_2^2 h_4 - 8 h2^3 h_4 - 6 h_1^2 h_3 h_4 - 20 h_1 h_2 h_3 h_4 \\&&
				- 12 h_2^2 h_3 h_4 - 6 h_1 h_3^2 h_4 - 4 h_2 h_3^2 h_4 - 5 h_1^2 h_4^2 - 24 h_1 h_2 h_4^2 - 24 h_2^2 h_4^2 - 16 h_1 h_3 h_4^2 - 24 h_2 h_3 h_4^2 \\&& - 5 h_3^2 h_4^2 - 10 h_1 h_4^3 - 20 h_2 h_4^3 - 10 h_3 h_4^3 - 5 h_4^4), \\
				q_9  (h_1, h_2, h_3, h_4) && \hspace*{-6mm}= -81 (h_1 + h_2 + h_4) (h_2 + h_3 + h_4) (4 h_1^2 h_2^2 + 8 h_1 h_2^3 + 4 h_2^4 - 4 h_1^2 h_2 h_3 + 4 h_1 h_2^2 h_3 + 8 h_2^3 h_3 \\&&
				- 5 h_1^2 h_3^2 - 4 h_1 h_2 h_3^2 + 4 h_2^2 h_3^2 - 4 h_1^2 h_2 h_4 - 12 h_1 h_2^2 h_4 - 8 h_2^3 h_4 + 2 h_1^2 h_3 h_4 - 4 h_1 h_2 h_3 h_4 - 12 h_2^2 h_3 h_4 \\&&
				+ 2 h_1 h_3^2 h_4 - 4 h_2 h_3^2 h_4 - 5 h_1^2 h_4^2 - 24 h_1 h_2 h_4^2 - 24 h_2^2 h_4^2 - 8 h_1 h_3 h_4^2 - 24 h_2 h_3 h_4^2 - 5 h_3^2 h_4^2 - 10 h_1 h_4^3 \\&&
				- 20 h_2 h_4^3 - 10 h_3 h_4^3 - 5 h_4^4). 
		\end{eqnarray*} }The only solution of polynomial equations $$q_1(h_1, h_2, h_3, h_4)  = q_2(h_1, h_2, h_3, h_4)  = \cdots  =q_9(h_1, h_2, h_3, h_4)  = 0$$ is $ (0,0,0,0)$ which implies that the associated variety of $L_{-14/3}(\g)$ does not contain any non-zero semisimple elements, so it is contained in the nilpotent cone of $\g$. The proof follows. \\
	\\
	\\
\end{proof}

\section{Proof of Theorem \ref{associated-variety}}
 \label{proof-thm1.3}

 \subsection{ The Subregular $W$-algebra $W^k(D_4, f_{sreg})$ }

 For a nilpotent element $f$  of  the Lie algebra $\g$, let $W^k(\g, f)$ be the universal $W$--algebra associated with $(\g, f)$ at level $k$.   It is obtained by generalized Drinfeld--Sokolov reduction (see \cite{KW04} for details):
 $$ W^k(\g, f) = H^0_{\DS,f} ( V^k (\g)). $$
 
 We  recall (cf. \cite{Ar2})  that for a nilpotent element $f$ and a vertex algebra $V$ which is a quotient of universal affine vertex algebra $V^k(\g)$ we  have 
\bea
X_{H^0_{\DS,f}(V)}=X_V\cap \Sf,
\qquad
H^0_{\DS,f}(V)\neq 0 \Longleftrightarrow f\in X_V, \label{property-Arakawa}
\eea 
where  $ \Sf =  f + \g ^e$  with $\g ^e = \{  x \in \g \ \vert \ [x,e] =0 \}$, is the Slodowy  slice associated to the $\mathfrak{sl}_2$ triple $(e,h, f)$.

The subregular nilpotent orbit  in $\g = D_4$ is the orbit corresponding to the partition
$
[5,3]$. 
An $\mathfrak{sl}_2$-triple for the subregular orbit of partition $[5,3]$ is (cf.   \cite{Rakotoarisoa}):
\[
e_{\sreg}
=
e_{\alpha_1}+e_{\alpha_3}+2e_{\alpha_4}+e_{\alpha_1+\alpha_2}-e_{\alpha_2+\alpha_4},
\]
\[
h_{\sreg}=4h_1+6h_2+4h_3+4h_4,
\]
\[
f_{\sreg}
=
-2f_{\alpha_1}+4f_{\alpha_3}+2f_{\alpha_4}+6f_{\alpha_1+\alpha_2}+6f_{\alpha_2+\alpha_3}.
\]

We fix the good grading associated with this orbit defined by
$
x_0=\frac{h_{sreg} }{2}.
$
 Then we have: 
 $$ \g  = \g_{-3} \oplus \g_{-2} \oplus \g_{-1} \oplus \g_{0} \oplus \g_1 \oplus \g_2 \oplus \g_3$$
 where
\[
  \g_1=\mathrm{span}\bigl\{
\ea{\varepsilon_1-\varepsilon_2},\ea{\varepsilon_1-\varepsilon_3},\ea{\varepsilon_2-\varepsilon_4},
\ea{\varepsilon_3-\varepsilon_4},\ea{\varepsilon_2+\varepsilon_4},\ea{\varepsilon_3+\varepsilon_4}
\bigr\}, \]
\[\g_{-1}=\mathrm{span}\bigl\{
\fa{\varepsilon_1-\varepsilon_2},\fa{\varepsilon_1-\varepsilon_3},\fa{\varepsilon_2-\varepsilon_4},
\fa{\varepsilon_3-\varepsilon_4},\fa{\varepsilon_2+\varepsilon_4},\fa{\varepsilon_3+\varepsilon_4}
\bigr\},
\]
\[
\g_2=\mathrm{span}\bigl\{
\ea{\varepsilon_1-\varepsilon_4},\ea{\varepsilon_1+\varepsilon_4},\ea{\varepsilon_2+\varepsilon_3}
\bigr\},
\qquad
\g_3=\mathrm{span}\bigl\{
\ea{\varepsilon_1+\varepsilon_2},\ea{\varepsilon_1+\varepsilon_3}
\bigr\},
\]
\[
\g_{-2}=\mathrm{span}\bigl\{
\fa{\varepsilon_1-\varepsilon_4},\fa{\varepsilon_1+\varepsilon_4},\fa{\varepsilon_2+\varepsilon_3}
\bigr\},
\qquad
\g_{-3}=\mathrm{span}\bigl\{
\fa{\varepsilon_1+\varepsilon_2},\fa{\varepsilon_1+\varepsilon_3}
\bigr\},
\]
\[
\g_0=\mathrm{span}\bigl\{
\hh{1},\hh{2},\hh{3},\hh{4},\,\ea{\varepsilon_2-\varepsilon_3},\,\fa{\varepsilon_2-\varepsilon_3}
\bigr\}.
\]

Let $ {\mathfrak m}= \g_1 \oplus \g_2 \oplus \g_3$. 
We define $\chi \in {\g}^{*}$ by 
$
\chi(x)=-(f_{\sreg}\mid x), \ x\in {\mathfrak m}$, 
with respect to the standard normalization
$
(\ea{\alpha}\mid \fa{\alpha})=1$. 
Then
\[
\chi(\ea{\varepsilon_1-\varepsilon_2})=2,\ 
\chi(\ea{\varepsilon_1-\varepsilon_3})=-6,\ 
\chi(\ea{\varepsilon_2-\varepsilon_4})=-6, \ 
\chi(\ea{\varepsilon_3-\varepsilon_4})=-4,\ 
\chi(\ea{\varepsilon_3+\varepsilon_4})=-2,
\]
and
\[
\chi(\ea{\varepsilon_1-\varepsilon_4})=
\chi(\ea{\varepsilon_2+\varepsilon_4})=
\chi(\ea{\varepsilon_1+\varepsilon_2})=
\chi(\ea{\varepsilon_1+\varepsilon_3})=
\chi(\ea{\varepsilon_1+\varepsilon_4})=
\chi(\ea{\varepsilon_2+\varepsilon_3})=0.
\]

Let $J_{\chi} = \sum_{ x \in {\mathfrak  m}} {\C} [\g ^{*}] ( x-\chi(x) )$. 

Denoting by $M$ the connected nilpotent subgroup with Lie algebra $\mathfrak{m}$  of the
adjoint group of $\g = D_4$, we have:
$$ R_{W^k(\g, f_{sreg}) } \cong  (S(\g) /  J_{\chi} ) ^M.$$




\begin{proposition}\label{prop:structure}
The universal subregular $W$-algebra $W^k(\g,f_{\sreg})$ has strong generating type
$
\W(2^3,3,4^2)$
 and central charge is
$
c(k)=-6\,\frac{(3k+14)(4k+17)}{k+6}.
$
In particular,
$
c\!\left(-\frac{14}{3}\right)=0$. 
\end{proposition}

\begin{proof}
 This result   appeared in   Table~2 of
\cite{AEM19}.   Here we include a proof for the reader's convenience.
 
By the PBW theorem of \cite{KRW03,KW04}, the strong generators are indexed by a
homogeneous basis of $\g^f$. For the subregular nilpotent $f= f_{sreg}$ in  $\g= D_4$ one has the
$\mathfrak{sl}_2$-decomposition
$
\g \cong 2V_6\oplus V_4\oplus 3V_2,
$ where $V_n$ denotes the irreducible $n+1$--dimensional  $\mathfrak{sl}_2$--module. This implies that 
$$
\dim \g^f_{-1}=3, \ 
\dim \g^f_{-2}=1,  \ 
\dim \g^f_{-3}=2. \
$$
So   $W^k(D_4,f_{\sreg})$ is strongly generated by three fields of conformal weight $2$, one of conformal weight $3$ and two of conformal weight $4$.
 
For the central charge we use the formula from   \cite{KRW03,KW04}:
\[
c(k)=\frac{k\,\dim\g}{k+h^\vee}-12k(x_0\mid x_0)-\sum_{j\ge 1}(12j^2-12j+2)\dim \g_j.
\]
Inserting 
$
\dim \g=28,  h^\vee=6,  (x_0\mid x_0)=6,
\dim \g_1=6,  \dim \g_2=3,\ \dim \g_3=2,
$
in the above formula, 
one obtains $c(k) = -6\,\frac{(3k+14)(4k+17)}{k+6}$.
 The proof follows.
\end{proof}

\begin{remark}
In the case where \(f_{[5,1^3]}\) is the hook nilpotent element corresponding to the
partition \([5,1^3]\), it was proved in \cite{AA+} that
\(k=-\frac{14}{3}\) is a collapsing level and that the simple \(W\)-algebra
\(W_k(\g, f_{[5,1^3]})\) collapses to \(L_{-\frac{4}{3}}(\mathfrak{sl}_2)\).
This approach cannot be applied directly to the subregular nilpotent element
\(f_{\sreg}\).  In particular, since the affine vertex subalgebra of
\(W^k(D_4,f_{\sreg})\) is trivial, while the conformal-weight two subspace is
three-dimensional, \cite[Criterion 2.5]{AA+} does not apply. We shall nevertheless prove that the level $k$  is collapsing and 
$
W_k(D_4,f_{\sreg})=\mathbb C .
$
Our proof uses the singular vectors constructed in this paper.
\end{remark}

\subsection{Projection of singular vectors in  $R_{W^k(\g, f_{sreg})}$}
 \label{subsec:reduced-ideal-symbols}

Let
\[
J=\langle v,\sigma(v),\sigma^2(v)\rangle \subset V^k(\g),
\qquad
\W:=W^k(D_4,f_{\sreg}).
\]
By exactness of Drinfeld--Sokolov reduction (cf. \cite{Ar3}), the ideal $J$ maps to an ideal
\[
I:=H^0_{\DS,f_{\sreg}}(J)\subset \W.
\]
\begin{proposition}\label{map-singular}
The ideal $I$ contains the whole conformal-weight two
subspace of $\W$.
\end{proposition}

\begin{proof}
 A computation using {\it Mathematica} and the expressions for
$v''$, $\sigma(v'')$ and $\sigma^2(v'')$ shows that in $S(\g)/J_\chi$ we have
\[
v''\equiv P_1,\qquad \sigma(v'')\equiv P_2,\qquad
\sigma^2(v'')\equiv P_3, 
\]
where $P_1, P_2, P_3 \in S(\g)$ are given by the following formulas, with $\nu=373248$:
\begingroup
\small
\setlength{\jot}{2pt}
\begin{align*}
P_1/\nu={}&
\frac{2}{9}\,\ea{\varepsilon_2-\varepsilon_3}^2
-\fa{\varepsilon_1-\varepsilon_2}+3\fa{\varepsilon_1-\varepsilon_3}
-\frac32\ea{\varepsilon_2-\varepsilon_3}\fa{\varepsilon_2-\varepsilon_3}
-3\fa{\varepsilon_2-\varepsilon_4} \\
&+6\fa{\varepsilon_2+\varepsilon_4}+3\fa{\varepsilon_3+\varepsilon_4}
-\frac14\hh{1}^2+\frac43\ea{\varepsilon_2-\varepsilon_3}\hh{2}
-3\fa{\varepsilon_2-\varepsilon_3}\hh{2}
-\frac12\hh{1}\hh{2} \\
&+\frac32\hh{2}^2+\frac13\ea{\varepsilon_2-\varepsilon_3}\hh{3}
-\frac14\hh{1}\hh{3}+\frac12\hh{2}\hh{3}
+\frac23\ea{\varepsilon_2-\varepsilon_3}\hh{4} \\
&-3\fa{\varepsilon_2-\varepsilon_3}\hh{4}
-\frac14\hh{1}\hh{4}+\frac32\hh{2}\hh{4}+\frac14\hh{3}\hh{4},\\[4pt]
P_2/\nu={}&
-\frac{4}{9}\,\ea{\varepsilon_2-\varepsilon_3}^2
+3\fa{\varepsilon_1-\varepsilon_2}-3\fa{\varepsilon_1-\varepsilon_3}
+\frac32\ea{\varepsilon_2-\varepsilon_3}\fa{\varepsilon_2-\varepsilon_3}
+3\fa{\varepsilon_2-\varepsilon_4} \\
&+6\fa{\varepsilon_2+\varepsilon_4}+2\fa{\varepsilon_3-\varepsilon_4}
-3\fa{\varepsilon_3+\varepsilon_4}
-\frac23\ea{\varepsilon_2-\varepsilon_3}\hh{1}
+\frac43\ea{\varepsilon_2-\varepsilon_3}\hh{2} \\
&-3\fa{\varepsilon_2-\varepsilon_3}\hh{2}
+\frac12\hh{1}\hh{2}-\frac32\hh{2}^2
-\frac14\hh{1}\hh{3}-\frac12\hh{2}\hh{3}-\frac14\hh{3}^2 \\
&+\frac23\ea{\varepsilon_2-\varepsilon_3}\hh{4}
-3\fa{\varepsilon_2-\varepsilon_3}\hh{4}
+\frac14\hh{1}\hh{4}-\frac32\hh{2}\hh{4}-\frac14\hh{3}\hh{4},\\[4pt]
P_3/\nu={}&
-\frac{2}{9}\,\ea{\varepsilon_2-\varepsilon_3}^2
+5\fa{\varepsilon_1-\varepsilon_2}+3\fa{\varepsilon_1-\varepsilon_3}
-\frac12\ea{\varepsilon_2-\varepsilon_3}\fa{\varepsilon_2-\varepsilon_3}
+3\fa{\varepsilon_2-\varepsilon_4} \\
&+3\fa{\varepsilon_2+\varepsilon_4}-4\fa{\varepsilon_3-\varepsilon_4}
+\fa{\varepsilon_3+\varepsilon_4}
-\frac23\ea{\varepsilon_2-\varepsilon_3}\hh{1}
-\frac12\hh{1}\hh{2} \\
&-\frac12\hh{2}^2+\frac13\ea{\varepsilon_2-\varepsilon_3}\hh{3}
+\frac14\hh{1}\hh{3}-\frac12\hh{2}\hh{3}
-\frac14\hh{1}\hh{4} \\
&-\frac12\hh{2}\hh{4}-\frac14\hh{3}\hh{4}-\frac14\hh{4}^2.
\end{align*}
\endgroup

Let $w_1,w_2,w_3\in I$ denote the images of
$v$, $\sigma(v)$ and $\sigma^2(v)$ under the Drinfeld--Sokolov reduction.
The three singular vectors in $V^k(\g)$ have conformal weight $6$ and have
  $\g$--highest weights $2\omega_1$, $2\omega_3$ and $2\omega_4$. Since
\[
(2\omega_1)(x_0)=(2\omega_3)(x_0)=(2\omega_4)(x_0)=4,
\]
their images in $W^k(\g,f_{\sreg})$ have conformal weight $6-4=2$.
Thus
\[
w_1,w_2,w_3\in I\cap \W_2.
\]

We shall identify these three elements in the Zhu $C_2$--algebra of $\W$.
We use the realization of the classical Drinfeld--Sokolov reduction
as the coordinate ring of the Slodowy slice (cf. 
\cite{DSK,Moreau-AV,Ar2}):
\[
R_{\W}\cong (S(\g)/J_\chi)^M\cong \C[S_{f_{\sreg}}].
\]
By a direct calculation,
$
(\g^{e_{\sreg}})_1=\operatorname{span}\{b_1,b_2,b_3\},
$
where
\[
b_1=-2\ea{\varepsilon_1-\varepsilon_2}
-\frac12\ea{\varepsilon_1-\varepsilon_3}
-\frac12\ea{\varepsilon_2-\varepsilon_4}
+\ea{\varepsilon_2+\varepsilon_4},
\]
\[
b_2=\ea{\varepsilon_1-\varepsilon_2}
+\frac12\ea{\varepsilon_1-\varepsilon_3}
-\frac12\ea{\varepsilon_2-\varepsilon_4}
+\ea{\varepsilon_3-\varepsilon_4}, \ 
b_3=-\ea{\varepsilon_1-\varepsilon_2}+\ea{\varepsilon_3+\varepsilon_4}.
\]
Let $x,y,z$ be the coordinate functions dual to the basis $b_1,b_2,b_3$;
thus a point of the Slodowy slice is written as
\[
s=f_{\sreg}+x(s) b_1+y(s) b_2+z(s) b_3+\cdots,
\]
where the dots denote the remaining homogeneous components of $\g^{e_{\sreg}}$.
Using the normalization $(\ea{\alpha}\mid \fa{\alpha})=1$, restriction to
$S_{f_{\sreg}}$ gives
$$
\fa{\varepsilon_1-\varepsilon_2}\mapsto -2x+y-z,
\qquad
\fa{\varepsilon_1-\varepsilon_3}\mapsto \frac{-x+y}{2}, \qquad 
\fa{\varepsilon_2-\varepsilon_4}\mapsto \frac{-x-y}{2},
$$
$$
\fa{\varepsilon_3-\varepsilon_4}\mapsto y, \qquad 
\fa{\varepsilon_2+\varepsilon_4}\mapsto x,
\qquad
\fa{\varepsilon_3+\varepsilon_4}\mapsto z.
$$
The functions $\hh{i}$, $\ea{\varepsilon_2-\varepsilon_3}$ and
$\fa{\varepsilon_2-\varepsilon_3}$ restrict to zero.
Therefore the images of $w_1,w_2,w_3$ in $R_{\W}\cong\C[S_{f_{\sreg}}]$ are
\[
\overline{w}_1=2\nu(4x+y+2z), \ 
\overline{w}_2=2\nu(y-3z), \ 
\overline{w}_3=\nu(-10x+y-4z).
\]

By Proposition~\ref{prop:structure}, we can choose  linearly independent conformal-weight two generators $G_1,G_2,G_3$ of $\W$ such that
their images in $R_{\W}\cong\C[S_{f_{\sreg}}]$ are
$
\overline{G}_1=x, \ 
\overline{G}_2=y, \
\overline{G}_3=z.
$

Since $\W$ has no conformal-weight one subspace, $C_2(\W)\cap \W_2=0$.
Hence the natural map $\W_2\to (R_{\W})_2$ is injective, and therefore in $\W_2$
we have
\[
w_1=2\nu(4G_1+G_2+2G_3), \ 
w_2=2\nu(G_2-3G_3), \ 
w_3=\nu(-10G_1+G_2-4G_3).
\]
This easily implies that   $w_1,w_2,w_3$ are linearly independent in
$\W_2$. By Proposition~\ref{prop:structure}, $\dim \W_2=3$, and therefore
$
I\cap \W_2=\W_2$. 
\end{proof}

\begin{remark}
Using the formula for the Virasoro field of affine \(W\)-algebras
\[
L=-\frac{1}{k+h^\vee}J^{\{f_{\sreg}\}}
\]
(cf. \cite[Theorem 3.2]{KW04}), one easily obtains
\[
L=\frac32 (G_1- G_2-2G_3).
\]
\end{remark}
 
\subsection{ Proof of Theorem \ref{associated-variety}}

 We shall first prove that $k$ is a collapsing level and
$W_k(\g,f_{\sreg}) \cong {\C}{\bf 1}$.
Set
$
\W:=W^k(\g ,f_{\sreg}),
I:=H^0_{\DS,f_{\sreg}}(J)$. 
By exactness of Drinfeld--Sokolov reduction (cf. \cite{Ar3})
\[
H^0_{\DS,f_{\sreg}}(L_k(\g))\cong \W/I.
\]

By Proposition~\ref{map-singular}, we have $I\cap \W_2=\W_2$.  Hence the
quotient $\W/I$ annihilates all generators of conformal weight $2$.  In particular, the
Virasoro vector $L$ belongs to $I$.

Let $X$ be any strong generator of positive conformal weight $d>0$. In the quotient
$\W/I$ one has $L =0$, hence
$
0= L(0) X=d\,X.
$
Since $d>0$, it follows that $X=0$ in $\W/I$.

Thus every strong generator of positive conformal weight vanishes in the quotient. The
quotient is generated by the vacuum vector alone, and therefore
$
\W/I\cong \C$. 
This proves  the assertion (1) of Theorem.


It remains to discuss the associated variety.   Applying (\ref{property-Arakawa}) to
$
V=L_k(\g)$, $ f=f_{\sreg},
$
and using the assertion (1), we obtain
$
f_{\sreg}\in X_{L_k(\g)}.
$
Moreover, since the reduction is one-dimensional, its associated variety is a point:
\[
X_{L_k(\g)}\cap \Sf=\{f_{\sreg}\}.
\]

On the other hand, by Theorem~\ref{quasi-lisse}, the associated variety
$X_{L_k(\g)}$ is contained in the nilpotent cone of $\g$. It is well known that for a
simple affine vertex algebra the associated variety is a closed $G$-stable conic subvariety
of $\g$ (see, for instance, \cite{Moreau-AV,ArMo-sheets}). Hence $X_{L_k(\g)}$ is a closed
$G$-stable nilpotent subvariety of $\g$.

Since $f_{\sreg}\in X_{L_k(\g)}$, we have
\[
\overline{{\Bbb O}}_{sreg}\subset X_{L_k(\g)}.
\]
Assume that the inclusion is strict. Since ${\Bbb O}_{sreg}$ is the subregular nilpotent
orbit in type $D_4$, the only nilpotent orbit strictly larger than ${\Bbb O}_{sreg}$ is the
regular orbit ${\Bbb O}_{reg}$. Hence
$
{\Bbb O}_{reg}\subset X_{L_k(\g)}.
$
By the standard dimension formula for Slodowy slices (see \cite{Slodowy}),
\[
\dim\bigl({\Bbb O}_{reg}\cap \Sf\bigr)=\dim {\Bbb O}_{reg}-\dim {\Bbb O}_{sreg}=2.
\]
Therefore \(X_{L_k(\g)}\cap \Sf\) would be positive-dimensional, contradicting
$
X_{L_k(\g)}\cap \Sf=\{f_{\sreg}\}$. 
This contradiction shows that
$
X_{L_k(\g)}=\overline{{\Bbb O}}_{sreg}.
$
This proves Theorem~\ref{associated-variety}.

 \section{Computational methods}

In the calculations   in this paper we use explicit symbolic computations  in 
\textsc{Mathematica}.

The first  and most important step is the construction of a singular vector
$
v\in V^{-14/3}(\g)
$
of conformal weight $6$ and  highest weight $2\omega_1$ with respect to $\g$. This is done by writing a
general vector in the relevant PBW weight space and imposing the singular vector condition. 
 The resulting linear system is solved in \textsc{Mathematica}. 
 The other two singular vectors
$
\sigma(v)$, $\sigma^2(v)
$
are then obtained by applying the order three Dynkin-diagram automorphism.

For classification of irreducible modules for $L_{-14/3}(\g)$  we  apply  Zhu's algebra theory.  Under the  isomorphism 
$
A\!\left(V^k(\g)\right)\cong U(\g),
$
the singular vectors map to explicit elements of the universal enveloping algebra. From
their adjoint $\g$-submodules we extract polynomial relations in the Cartan variables,
which are then solved in \textsc{Mathematica} in order to determine the irreducible
highest weights in the category $\mathcal O$.

To study the associated variety, we work in Zhu's $C_2$-algebra
$
R_{V^k(\g)}\cong S(\g).
$
The singular vectors are mapped  to  elements
$
v''$, $\sigma(v'')$, $\sigma^2(v'')$ 
in the symmetric algebra. Applying the Chevalley projection produces polynomial relations
on the Cartan subalgebra. These computations, again carried out in \textsc{Mathematica},
show that the associated variety contains no nonzero semisimple element, and therefore is
contained in the nilpotent cone.

Finally, for the subregular Drinfeld--Sokolov reduction, we fix a subregular
$\mathfrak{sl}_2$-triple \\
$
(e_{sreg},h_{sreg},f_{sreg})
$
and the corresponding character
$
\chi(x)=-(f_{sreg}\mid x), x\in \g_{>0}.
$
We then  calculate the images of vectors 
$
v'',\sigma(v''), \sigma^2(v'')
$
in the  $S(\g) \ (\mbox{mod} \ J_{\chi})$. This produces three non zero 
polynomials
$
P_1, P_2,  P_3$, 
whose coefficients are computed in \textsc{Mathematica}.  Next we restrict these polynomials to the Slodowy slice $S_{f_{sreg}}$, which enables us to identify the images of the singular vectors in $V^k(\g)$ with generators of conformal weight $2$ of the affine $W$-algebra $W^k(\g,f_{sreg})$.

\end{document}